\documentclass[12 pt, reqno]{amsart}

\usepackage{amsmath}
\usepackage{amsfonts}
\usepackage{amsthm}
\usepackage{enumerate}
\usepackage{cleveref}
\usepackage{amssymb}
\usepackage{setspace}
\usepackage{comment}
\usepackage{fancyhdr}
\usepackage{url}
\usepackage[english]{babel}
\usepackage[english=american]{csquotes}
 \newtheorem{thm}{Theorem}[section]
 \newtheorem{cor}[thm]{Corollary}
 \newtheorem{lem}[thm]{Lemma}
 \newtheorem{prop}[thm]{Proposition}
 \theoremstyle{definition}
 
 \theoremstyle{remark}
 \newtheorem{rem}[thm]{Remark}
 
 \numberwithin{equation}{section}
\newcommand{\sgn}{\operatorname{sgn}}
\newcommand{\ind}{\operatorname{ind}}

\begin{document}

\title{Fredholm operators in the Toeplitz Algebra $\mathcal{I}(QC)$}

\author{Adam Orenstein}

\address{
244 Mathematics Building\\
University at Buffalo\\
Buffalo, NY 14260}

\email{adamoren@buffalo.edu}



\date{\today}


\begin{abstract}
We will give a complete description of $\mathcal{I}$, the set of invertible quasicontinuous functions on the unit circle.  After doing this, we will then classify the path-connected components of $\mathcal{I}$ and show that $\mathcal{I}$ has uncountably many path-connected components.  We will then use the above classifications to characterize $\mathcal{F}$, the set of Fredholm operators of the C$^*$-algebra generated by the Toeplitz operators $T_\phi$ with quasicontinuous symbols $\phi$.  Then we will classify the path-connected components of $\mathcal{F}$ and show that $\mathcal{F}$ also has uncountably many path-connected components.
\end{abstract}

\maketitle

\section{Introduction}
Let $\mathbb{T}$ be the unit circle with the Lebesgue measure.  Let $C(\mathbb{T})$ be the set of all continuous functions on $\mathbb{T}$.  For any $n\in\mathbb{Z}$, let $\chi_n:\mathbb{T}\rightarrow\mathbb{T}$ be defined by \[\chi_n(t)=e^{int}.\]  For $p=1,2,\infty,$ let $H^p(\mathbb{T})$ be defined by \[H^p(\mathbb{T})=\left\{f\in L^p(\mathbb{T}): \int_0^{2\pi}f(t)\chi_n(t)dt=0 \text{ for all } n>0\right\}.\]  The spaces $H^p(\mathbb{T})$ are called the Hardy spaces on $\mathbb{T}$.  It is shown in \cite[page 147]{Doug} as proposition 6.36 that $H^\infty+C(\mathbb{T})$ is a Banach subalgebra of $L^\infty\left(\mathbb{T}\right)$.  Let $QC$ be the set defined by \[QC=\left[H^\infty+C(\mathbb{T})\right]\bigcap \overline{\left[H^\infty+C(\mathbb{T})\right]}\]  where $\overline{H^\infty+C(\mathbb{T})}=\{\overline{f}:f\in H^\infty+C(\mathbb{T})\}$ \cite[page 157]{Doug}.  A function in $QC$ is called quasicontinuous.  It is easy to show that $QC$ is a commutative Banach subalgebra of $L^\infty(\mathbb{T})$.

Let $\mathfrak{L}(H^2)$ be the set of all bounded linear operators on $H^2$.  Denote the operator norm on $\mathfrak{L}(H^2)$ by $\|\cdot\|$.  For any $f\in L^\infty\left(\mathbb{T}\right)$, let $T_f$ be the operator on $H^2$ defined by \[T_f(g)=P(fg) \text{ for all }g\in H^2\] where $P$ is the orthogonal projection of $L^2(\mathbb{T})$ onto $H^2$.  $T_f$ is called the Toeplitz operator on $H^2$ with symbol $f$.  Since $f\in L^\infty(\mathbb{T})$, $T_f\in\mathfrak{L}(H^2)$ and $\|T_f\|=\|f\|_\infty$ \cite[page 160]{Doug}.  For any subset $S$ of $L^\infty(\mathbb{T})$, let $\mathcal{I}(S)$ be the smallest closed subalgebra of $\mathfrak{L}(H^2)$ containing $\{T_f:f\in S\}$.  $\mathcal{I}(S)$ is called a Toeplitz algebra.

Toeplitz algebras and Toeplitz operators on $H^2$ with quasicontinuous symbols have been studied by many different people in the literature.  In particular the path-connected components of the Fredholm operators in $\mathcal{I}(C(\mathbb{T}))$ have been completely determined.  The purpose of this paper is to do the same for $\mathcal{I}(QC)$.  In doing this, we will have an example of a Toeplitz algebra whose Fredholm operators have uncountably many path-connected components.  Until this paper, no such example has been published.

\section{Preliminaries}
\subsection{More Notation}\label{notat}

Here we will establish some more notation.  Let $\mathbb{D}$ be the open unit disk in $\mathbb{C}$.  For any $D\subseteq \mathbb{T}$, let $|D|$ denote the Lebesgue measure of $D$ on $\mathbb{T}$.  Let $C_{\mathbb{R}}(\mathbb{T})$ be the subset of $C(\mathbb{T})$ consisting of all real-valued functions, $L_{\mathbb{R}}^\infty(\mathbb{T})$ be the subset of $L^\infty(\mathbb{T})$ consisting of all real-valued functions and $QC_\mathbb{R}$ denote the set of all real-valued functions in $QC$.  We will also use the following conventions: for any two sets $Q$ and $B$ and for any function $g$, let $QB=\{cb:c\in Q, b\in B\}$, $gQ=\{g\}Q$, $Q\setminus B=\{q\in Q: q\notin B\}$, $Q+B=\{q+b:q\in Q, b\in B\}$ and $\exp(Q)=\{\exp(q):q\in Q\}$.

\subsection{Important Functions}

Now we will define some important functions that we will use.  Let $w\in L^1(\mathbb{T})$.  Then following \cite[page 56 and page 81]{krantz} and \cite[page 255]{zhuOp} respectively, we define $\widetilde{w}:\mathbb{T}\rightarrow\mathbb{C}$ and $\widehat{w}:\mathbb{D}\rightarrow\mathbb{C}$ by \begin{equation}\label{hilDef}\widetilde{w}(\theta)=\sum_{n\in\mathbb{Z}}(-i\sgn(n))a_w(n)\chi_n(\theta)\end{equation} where \begin{equation}\label{coeff}a_w(n)=\frac{1}{2\pi}\int_0^{2\pi} w(t)\chi_{-n}(t)dt \hspace{.2cm} \text{ for all } n\in\mathbb{Z},\end{equation} and \begin{equation}\label{poiDef}\widehat{w}(z)=\frac{1}{2\pi}\int_0^{2\pi} w(t)\text{Re}\left[\frac{e^{it}+z}{e^{it}-z}\right]dt.\end{equation}  We will call $\widetilde{w}$ the Hilbert Transform of $w$ and $\widehat{w}$ the Poisson transform of $w$.

Another important function we will occasionally use is $\mathcal{C}(w):\mathbb{D}\cup \mathbb{T}\rightarrow \mathbb{C}$ defined by \begin{equation}\label{conjDef}\begin{split}&\mathcal{C}(w)(z)=\frac{1}{2\pi}\int_0^{2\pi} w(t)\text{Im}\left[\frac{e^{it}+z}{e^{it}-z}\right]dt \text{ for all } z\in \mathbb{D}\\& \mathcal{C}(w)(e^{i\theta})=\lim_{r\nearrow1}\mathcal{C}(w)(re^{i\theta}) \text{ for all } \theta\in[0,2\pi)\end{split}\end{equation}  \cite[page 256]{zhuOp}.  As in \cite[page 256]{zhuOp}, we will call $\mathcal{C}(w)$ the conjugation operator of $w$ and the mapping $w \mapsto \mathcal{C}(w)$ the conjugation operator.

\subsection{Important Spaces}\label{imSpaces}

In this section, we will define the spaces of functions that we will be working with.  For any $f\in L^1(\mathbb{T})$ and any subarc $I$ of $\mathbb{T}$, let \[f_I=\frac{1}{|I|}\int_I f(\theta)d\theta.\]  For any $f\in L^2(\mathbb{T})$, define $\|f\|_{BMO}$ by \[\|f\|_{BMO}= \sup\left\{\left[\frac{1}{|I|}\int_I\left|f(\theta)-f_I\right|^2d\theta\right]^{\frac{1}{2}}:I \text{ is any subarc of }\mathbb{T}\right\}.\]  Let $BMO$ and $VMO$ be defined by \[BMO=\left\{f\in L^2(\mathbb{T}):\|f\|_{BMO}<\infty\right\}\] and \begin{equation}\label{VMOdef}VMO=\left\{f\in BMO:\lim_{|I|\rightarrow 0}\left[\frac{1}{|I|}\int_I\left|f(\theta)-f_I\right|^2d\theta\right]=0\right\}.\end{equation}  We say $BMO$ is the space of all functions in $L^2(\mathbb{T})$ that have bounded mean oscillation on $\mathbb{T}$ and $VMO$ is the space of all functions in $BMO$ that have vanishing mean oscillation on $\mathbb{T}$.  Here our definitions of $BMO$ and $VMO$ are taken from \cite[page 266 and page 275]{zhuOp}.  We will also denote the set of all real-valued functions in $VMO$ by $VMO_{\mathbb{R}}$.  In fact \begin{equation}\label{quasi} QC=VMO\cap L^\infty\left(\mathbb{T}\right)\end{equation} \cite[page 377]{Garn}.

Note that some people use the conjugation operator instead of the Hilbert transform when working with $BMO$ and $VMO$.  However for every $f\in L^1\left(\mathbb{T}\right)$, Lemma 1.2 \cite[page 103]{Garn} and Theorem 1.6.11 \cite[page 87]{krantz} together yield \begin{equation}\label{conjHil} -\widetilde{f}=\mathcal{C}(f) \text{ [a.e] on } \mathbb{T}.\end{equation}

Moreover as proved in \cite[page 277]{zhuOp}, \[VMO=C(\mathbb{T})+\mathcal{C}\left(C(\mathbb{T})\right).\]  It is also easy to see that $f\in L^1(\mathbb{T})$ and real-valued implies $\widetilde{f}$ is real-valued.  Then by \eqref{conjHil} and some straightforward calculations, \begin{equation}\label{vmoReal} VMO=C(\mathbb{T})+\widetilde{C(\mathbb{T})} \text{ and }VMO_\mathbb{R}=C_\mathbb{R}(\mathbb{T})+\widetilde{C_\mathbb{R}(\mathbb{T})}. \end{equation}

\section{Fredholm Index}\label{SectInd}

Let $\mathfrak{LC}(H^2)$ be the set of all compact operators on $H^2$ and $\mathfrak{F}(H^2)$ be the set of all Fredholm operators on $H^2$.  For any $A\in\mathfrak{L}(H^2)$ let $A^*$ be the adjoint of $A$.  Let $j:\mathfrak{F}(H^2)\rightarrow\mathbb{Z}$ be defined by\[j(A)=\dim(\ker(A))-\dim(\ker(A^*)).\]  Then $j$ is surjective, continuous with respect to the operator norm and for any $A,B\in \mathfrak{F}(H^2)$ and $K\in\mathfrak{LC}(H^2)$, we have $j(AB)=j(A)+j(B)$ and $j(A+K)=j(A)$ \cite[page 123]{Doug}.

Let \[(H^\infty+C(\mathbb{T}))^{-1}=\{f\in H^\infty+C(\mathbb{T}):f^{-1}\in H^\infty+C(\mathbb{T})\}.\]  Following \cite[page 169]{Doug} and \cite{sarason}, we will work with the function $\ind:(H^\infty+C(\mathbb{T}))^{-1}\rightarrow \mathbb{Z}$ defined by \[\ind(f)=n\left(\hat{f}_r\left(e^{i\theta}\right),0\right)\] where $\hat{f}_r\left(e^{i\theta}\right)=\widehat{f}(re^{i\theta})$ for all $\theta\in[0,2\pi)$ and $n\left(\hat{f}_r(e^{i\theta}),0\right)$ is the winding number of $\hat{f}_r\left(e^{i\theta}\right)$ about 0 for all $r>0$ such that $r_0<r<1$ where $r_0=r_0(f)>0$ is as in Theorem 4 \cite{sarason}.  By both Corollary 7.34 \cite[page 168]{Doug} and Theorem 7.36 \cite[page 169]{Doug}, \begin{equation}\label{index} j(T_f)=-\ind(f) \text{ for all } f\in (H^\infty+C(\mathbb{T}))^{-1}.\end{equation}  It follows that $\ind$ is continuous with respect to the norm $\|\cdot\|_\infty$.  We will prove below that $\ind(fg)=\ind(f)+\ind(g)$ for all $f,g\in (H^\infty+C(\mathbb{T})^{-1}$.

\section{Main Results}\label{motivQC} Recall above that $\mathfrak{I}(QC)$ be the $C^*$-algebra generated by $\{T_\phi:\phi\in QC\}$ and \[\mathcal{I}=\{f\in QC:f^{-1}\in QC\}.\]   and $\mathcal{F}$ be defined by \[\mathcal{F}=\mathfrak{F}(H^2)\cap\mathfrak{I}(QC).\]  In this paper we will completely classify $\mathcal{I}$ and $\mathcal{F}$, and to give a complete description of the path-connected components of $\mathcal{I}$ and of $\mathcal{F}$.  More specifically, we will prove the following three theorems:

\begin{thm}\label{classI} \[\mathcal{I}=\{\chi_n\}_{n\in\mathbb{Z}}\exp\left(QC_\mathbb{R}\right)\exp\left(iVMO_{\mathbb{R}}\right).\]

\end{thm}

\begin{thm}\label{mainI}
For any $h\in VMO_{\mathbb{R}}$, let \begin{equation}\label{pPath} P_h=\{\exp(ig): g\in QC_{\mathbb{R}}+\{h\}\}\end{equation}
and let \begin{equation}\label{uniPath} Q_h=\exp(QC_{\mathbb{R}})P_h. \end{equation}  Then the path-connected components of $\mathcal{I}$ are $\{\chi_kQ_g\}_{\left\{\begin{subarray}{l} k\in\mathbb{Z} \text{ and either }\\ g\in VMO_{\mathbb{R}}\setminus L_\mathbb{R}^\infty \text{ or }\\ g=0\end{subarray}\right\}}$.
\end{thm}

\begin{thm}\label{mainF}
With the same notation from Theorem \ref{mainI}, let \begin{equation}\label{fredPath} {_m}V_h=\{T_f+K\in \mathcal{F}:f\in\chi_{-m}Q_{h}, K\in\mathfrak{LC}(H^2)\}.\end{equation}  Then the path-connected components of $\mathcal{F}$ are $\{{_m}V_h\}_{\left\{\begin{subarray}{l} m\in\mathbb{Z} \text{ and either } \\ h\in VMO_{\mathbb{R}}\setminus L_\mathbb{R}^\infty \text{ or }\\ h=0\end{subarray}\right\}}$.

\end{thm}

In addition to the above theorems, we will also prove the following:

\begin{cor}\label{uncountable}
$\mathcal{I}$ and $\mathcal{F}$ both have uncountably many path-connected components.
\end{cor}

\section{Classification of $\mathcal{I}$}
\subsection{Important lemmas}\label{Prelem}

First we will show $\ind(fg)=\ind(f)+\ind(g)$ for all $f,g\in (H^\infty+C(\mathbb{T}))^{-1}$.  To prove this, we will need the following two propositions found in \cite[page 159 and page 164]{Doug}.

\begin{prop}\label{commToe}
If $\phi\in L^\infty$ and $\psi,\overline{\theta}\in H^\infty$, then $T_{\phi}T_{\psi}=T_{\phi\psi}$ and $T_{\theta}T_\phi=T_{\theta\phi}$.
\end{prop}

\begin{prop}\label{ToeCommCompact}
If $\phi\in C(\mathbb{T})$ and $\psi\in L^\infty$, then $T_{\phi}T_{\psi}-T_{\phi\psi},T_{\psi}T_{\phi}-T_{\psi\phi}\in\mathfrak{LC}(H^2)$.
\end{prop}

\begin{lem}\label{indSum}
For any $f,g\in (H^\infty+C(\mathbb{T}))^{-1}$, $\ind(fg)=\ind(f)+\ind(g)$.
\end{lem}

\begin{proof}
Let $f,g\in (H^\infty+C(\mathbb{T}))^{-1}$.  So we may write $f=h_f+p_f$ and $g=h_g+p_g$ where $h_f,h_g\in H^\infty$ and $p_f,p_g\in C(\mathbb{T})$.  Then by Proposition \ref{commToe} and Proposition \ref{ToeCommCompact}, there exist $L_1,L_2\in\mathfrak{LC}(H^2)$ such that \[\begin{split}T_fT_g&=T_{h_f}T_{h_g}+T_{h_f}T_{p_g}+T_{p_f}T_{h_g}+T_{p_f}T_{p_g}\\&=T_{h_fh_g}+T_{h_fp_g}+L_1+T_{p_fh_g}+T_{p_fp_g}+L_2\\&= T_{h_fh_g+h_fp_g+p_fh_g+p_fp_g}+L_1+L_2\\&=T_{fg}+L_1+L_2.\end{split}\]  Thus $T_fT_g=T_{fg}+J$ for some $J\in\mathfrak{LC}(H^2)$.  Then by Section \ref{SectInd} $j(T_{fg})=j(T_fT_g)=j(T_f)+j(T_g)$.  Therefore by \eqref{index}, $\ind(fg)=\ind(f)+\ind(g)$.

\end{proof}

We will also need the following lemmas in order to prove Theorem \ref{classI}.  Some of these results will also be used to prove Theorem \ref{mainI} and Theorem \ref{mainF}.  The first lemma will be used many times. 

\begin{lem}\label{exp}
Let $\mathfrak{B}$ be a Banach algebra, $G$ be the collection of all the invertible elements in $\mathfrak{B}$ and $G_0$ be the connected component in $G$ which contains the identity.  If $\mathfrak{B}$ is commutative then $G_0=\exp(\mathfrak{B})$.
\end{lem}

From now on for convenience, we will denote the path-connected component in $\mathcal{I}$ containing 1 by $G_0$.  As we will see, being able to calculate $\widetilde{\widetilde{f}}$ for any $f\in H^\infty+C(\mathbb{T})$ is necessary.  The next lemma shows us how to do this.

\begin{lem}\label{DblHil}
For every $f\in L^2(\mathbb{T})$, $\widetilde{\widetilde{f}}=-f+a_f(0)$.

\end{lem}

\begin{proof}
Let $f\in L^2(\mathbb{T})$.  Then by Theorem 1.6.16 \cite[page 87]{krantz}, \newline$f=\sum_{n\in\mathbb{Z}} a_f(n)\chi_n$ [a.e] on $\mathbb{T}$.  Also from \eqref{hilDef} and Parseval's identity, $\|\widetilde{f}\|_2^2 = \sum_{n=-\infty}^\infty |-i\sgn(n)a_f(n)|^2 \leq \sum_{n=-\infty}^\infty |a_f(n)|^2=\|f\|_2^2$ where $\|\cdot\|_2$ is the norm on $L^2(\mathbb{T})$.  It follows that the Hilbert Transform is both bounded and linear on $L^2(\mathbb{T})$.  Hence \[\begin{split}\widetilde{\widetilde{f}}=\sum_{n\in\mathbb{Z}}(-i\sgn(n))^2a_f(n)\chi_n&=-\sum_{n\in\mathbb{Z}}(\sgn(n))^2a_f(n)\chi_n\\& =-\sum_{n\in\mathbb{Z}}a_f(n)\chi_n +a_f(0).\end{split}\]  Therefore $\widetilde{\widetilde{f}}=-f+a_f(0)$.
\end{proof}

An immediate consequence of Lemma \ref{DblHil} and \eqref{vmoReal} is \begin{equation}\label{doublevmo} \widetilde{VMO}\subseteq VMO \text{ and } \widetilde{VMO_\mathbb{R}}\subseteq VMO_\mathbb{R}.\end{equation}

The next lemma will be crucial in classifying $\mathcal{I}$.

\begin{lem}\label{I}
Let $\phi\in QC$.  Then $\phi\in \mathcal{I} \text{ if and only if } \phi, \overline{\phi}\in (H^\infty+C(\mathbb{T}))^{-1}$.
\end{lem}

\begin{proof}
$(\Rightarrow)$  Suppose $\phi\in \mathcal{I}$.  Then $\phi g=1$ for some $g\in QC$.  Thus $\overline{\phi}\overline{g}=1$.  Since $QC$ is self-adjoint $\overline{\phi},\overline{g}\in QC$.  Therefore since $QC\subseteq H^\infty+C(\mathbb{T})$ we have $\phi, \overline{\phi}\in (H^\infty+C(\mathbb{T}))^{-1}$.

$(\Leftarrow)$ Assume $\phi, \overline{\phi}\in (H^\infty+C(\mathbb{T}))^{-1}$.  Then for some functions $g,h\in H^\infty+C(\mathbb{T})$, $\phi g=\overline{\phi}h=1$.  This means $\overline{\phi}\overline{g}=\overline{\phi}h$.  Thus $\overline{g}=h$ and $g\in \left[H^\infty+C(\mathbb{T})\right]\bigcap \overline{\left[H^\infty+C(\mathbb{T})\right]}$.  So $g\in QC$.  Therefore $\phi\in \mathcal{I}$.
\end{proof}

From the lemma, we can see that to fully classify $\mathcal{I}$, we must first classify $(H^\infty+C(\mathbb{T}))^{-1}$.  To do this, we will need Lemma \ref{inHardy} and the following Theorem from \cite{sarason}.

\begin{thm}\label{uniModInv}
Let $n\in\mathbb{Z}$ and $w$ be a unimodular function in $(H^\infty+C(\mathbb{T}))^{-1}$ satisfying $\ind(w)=n$.  Then for some $u,v\in C_{\mathbb{R}}(\mathbb{T})$, $w=\chi_n\exp(i(u+\mathcal{C}(v)))$.
\end{thm}

\begin{lem}\label{inHardy}
Let $w\in L_{\mathbb{R}}^\infty$.  Then $\exp(w-i\widetilde{w})$ is invertible in $H^\infty$.
\end{lem}

\begin{proof}
Let $G:\mathbb{D}\rightarrow\mathbb{C}$ be defined by \[G(re^{it})=\exp\left(\widehat{w}(re^{it})+i\mathcal{C}\left(w\right)(re^{it})\right). \]  Then from Theorem 11.32 \cite[page 249]{Rud} and Theorem 17.16 \cite[page 343]{Rud}, $G$ is analytic on $\mathbb{D}$ and $\lim_{r\nearrow 1}G(re^{it})$ exists for almost every $e^{it}\in\mathbb{T}$.  Let $G^*:\mathbb{T}\rightarrow \mathbb{C}$ be defined by
\[G^*(e^{it})=\left\{
\begin{array}{cc}
\lim_{r\nearrow 1}G(re^{it})& \mbox{ if } \lim_{r\nearrow 1}G(re^{it}) \mbox{ exists}\\
0  & \mbox{otherwise}
\end{array}\right.\]  Then by Theorem 11.23 \cite[page 244]{Rud}, \eqref{conjDef} and \eqref{conjHil} above \[G^*=\exp(w-i\widetilde{w}) \text{ almost everywhere on }\mathbb{T}.\]  It follows that $G^*\in L^\infty(\mathbb{T})$.  Moreover by the proof of Theorem 11.32 \cite[page 249]{Rud}, $G=\widehat{G^*}$.  Thus by Lemma 6.43 \cite[page 150-151]{Doug} \[G(re^{i\theta})=\sum_{n\in\mathbb{Z}}^\infty a_{G^*}(n)r^{|n|}\chi_n \text{ for all }z\in\mathbb{D}.\]  Since $G$ is analytic on $\mathbb{D}$, we must have $a_{G^*}(n)=0$ for all $n<0$.  Hence $G^*\in H^\infty$.  Since $w\in L_{\mathbb{R}}^\infty$ was arbitrary, the above argument also implies $\exp(-w+i\widetilde{w})\in H^\infty$.  Therefore $\exp(w-i\widetilde{w})$ is invertible in $H^\infty$.\end{proof}

Now we will classify $(H^\infty+C(\mathbb{T}))^{-1}$.
\begin{prop}\label{saras}
Let $f\in L^\infty(\mathbb{T})$.  Then \[f\in (H^\infty+C(\mathbb{T}))^{-1} \text{ if and only if } f=\chi_n\exp(ig)\exp(w-i\widetilde{w})\] for some $n\in\mathbb{Z}$, $w\in L_\mathbb{R}^\infty$ and $g\in VMO_{\mathbb{R}}$.
\end{prop}

\begin{proof}
Let $g\in VMO_\mathbb{R}$, $n\in\mathbb{Z}$ and $w\in L_\mathbb{R}^\infty$.  Then by \eqref{vmoReal} $g=h-\widetilde{u}$ for some $u,h\in C_\mathbb{R}(\mathbb{T})$.  Thus $\exp(ig)=\exp(ih)\exp(u-i\widetilde{u})\exp(-u)$  Then from Lemma \ref{inHardy}, $\exp(w-i\widetilde{w})\exp(ih)\exp(u-i\widetilde{u})\exp(-u)\in (H^\infty+C(\mathbb{T}))^{-1}$.  Hence $\chi_n\exp(ig)\exp(w-i\widetilde{w})\in (H^\infty+C(\mathbb{T}))^{-1}$.

Assume $f\in (H^\infty+C(\mathbb{T}))^{-1}$.  It follows from the continuous functional calculus \cite[page 62]{zhuAl} that $\ln|f|\in L_\mathbb{R}^\infty$ and thus by Lemma \ref{inHardy} \newline$\exp\left[\ln|f|-i\left(\widetilde{\ln|f|}\right)\right]\in (H^\infty+\mathcal{C}(\mathbb{T}))^{-1}$.  It is also true that \\ $\left|\exp\left[\ln|f|-i\left(\widetilde{\ln|f|}\right)\right]\right|=|f|$.  Hence by Theorem \ref{uniModInv}, \eqref{conjHil} and \eqref{vmoReal}, $f\exp\left[-\ln|f|+i\left(\widetilde{\ln|f|}\right)\right]=\chi_n\exp(ig)$ for some $n\in\mathbb{Z}$ and $g\in VMO_{\mathbb{R}}$.  Therefore $f=\chi_n\exp(w-i\widetilde{w})\exp(ig)$ for some $n\in\mathbb{Z}$, $w\in L_\mathbb{R}^\infty$ and $g\in VMO_{\mathbb{R}}$.
\end{proof}

We will now prove that the only integer valued functions in $VMO$ are the constant functions.  Recall that for any measurable function $g:\mathbb{T}\rightarrow\mathbb{C}$, the essential range of $g$ is the set $R_g$ defined by \[R_g=\{\lambda\in\mathbb{C}:\left|\{e^{it}\in\mathbb{T}:|g(e^{it})-\lambda|<\epsilon\}\right|>0 \text{ for every }\epsilon>0\}.\]

\begin{lem}\label{essRan}
If $g\in VMO_\mathbb{R}$, then $R_g$ is a connected subset of $\mathbb{R}$.
\end{lem}

\begin{proof}
From the definition of $R_g$, we can see that $g(\mathbb{T})\subseteq \mathbb{R}$ implies $R_g\subseteq \mathbb{R}$.  Thus it suffices to show if $a,b\in R_g$ with $a<b$ and $a<c<b$, then $c\in R_g$.

Assume $c\notin R_g$.  Then for some $\gamma>0$, $|g^{-1}(c-\gamma,c+\gamma)|=0$.  Choose $0<\epsilon\leq \gamma$ so that $a< c-\epsilon$ and $c+\epsilon<b$.  Let $A=g^{-1}\left((-\infty,c-\epsilon]\right)$ and $B=g^{-1}\left([c+\epsilon,\infty)\right)$.  Clearly $B\subseteq \mathbb{T}\setminus A$.  Then $a,b\in R_g$ implies $|A|>0$ and $|\mathbb{T}\setminus A|>0$.  Let $\omega_A:\mathbb{T}\rightarrow\{0,1\}$ be defined by

\[\omega_A(x)=
\left\{
  \begin{array}{cc}
   1 & \mbox{ if } x\in A\\
   0 & \mbox{if } x\notin A\\
\end{array}
\right.\]

and let $I$ be any subarc of $\mathbb{T}$.  Since $\frac{1}{|I|}\int_{I}\omega_A=\frac{|I\cap A|}{|I|}$, we have \[\begin{split}&\frac{1}{|I|}\int_{I}\left|\omega_A-\frac{|I\cap A|}{|I|}\right|^2=\\&=\frac{1}{|I|}\int_{I\cap A}\left|\omega_A-\frac{|I\cap A|}{|I|}\right|^2+\frac{1}{|I|}\int_{I\cap (\mathbb{T}\setminus A)}\left|\omega_A-\frac{|I\cap A|}{|I|}\right|^2\\&=\frac{1}{|I|}\int_{I\cap A}\left|1-\frac{|I\cap A|}{|I|}\right|^2+\frac{1}{|I|}\int_{I\cap (\mathbb{T}\setminus A)}\left|0-\frac{|I\cap A|}{|I|}\right|^2\\&=\frac{|I\cap A||I\cap (\mathbb{T}\setminus A)|\left(|I\cap A|+|I\cap (\mathbb{T}\setminus A)|\right)}{|I|^3}\\&=\frac{|I\cap A||I\cap (\mathbb{T}\setminus A)|}{|I|^2}\\&=\frac{|I\cap A||I\cap B|}{|I|^2}\end{split}\] where the latter equality follows from the fact that $|\mathbb{T}\setminus A|=|B|+|g^{-1}\left((c-\epsilon,c+\epsilon)\right)|=|B|$.  By proposition 5.3 \cite[page 461]{XiaVmo}, \\$\omega_A\notin VMO_\mathbb{R}$.  So for some subarc $J$ of $\mathbb{T}$, $$\lim_{|J|\rightarrow0}\frac{|J\cap A||J\cap B|}{|J|^2}\neq0.$$

On the other hand direct calculation yields \[\frac{1}{|I|}\int_I|g-g_I|^2=\frac{1}{2|I|^2}\int_I\int_I|g(t)-g(x)|^2dxdt\] for any subarc $I$ of $\mathbb{T}$ \cite[page 276]{zhuOp}.  Thus \[\begin{split}\lim_{|I|\rightarrow0}\frac{1}{|I|}\int_I|g-g_I|^2&\geq \lim_{|I|\rightarrow0}\frac{1}{2|I|^2}\int_{I\cap A}\int_{I\cap B}|g(t)-g(x)|^2dxdt\\&\geq \lim_{|I|\rightarrow0}\frac{(2\epsilon)^2|I\cap A||I\cap B|}{2|I|^2}.\end{split}\]  So since $g\in VMO_\mathbb{R}$, $\lim_{|J|\rightarrow0}\frac{4\epsilon^2|J\cap A||J\cap B|}{2|J|^2}=0$ and which implies $$\lim_{|J|\rightarrow0}\frac{|J\cap A||J\cap B|}{|J|^2}=0.$$  This is a contradiction by the above statements.  Therefore $c\in R_g$ and $R_g$ is connected.

\end{proof}
Hence we get the following corollary.

\begin{cor}\label{Z}
If $f\in VMO_{\mathbb{R}}$ and integer valued, then $f$ is constant.
\end{cor}

\subsection{Proof of Theorem \ref{classI}}

\begin{proof}[Proof of Theorem \ref{classI}]

By Lemma \ref{exp}, Lemma \ref{I} and Proposition \ref{saras}, \[\{\chi_n\}_n\exp\left(QC_\mathbb{R}\right)\exp\left(iVMO_{\mathbb{R}}\right)\subseteq \mathcal{I}.\]

Let $\phi\in\mathcal{I}$.  Then by Lemma \ref{I}, $\phi,\overline{\phi}\in (H^\infty+C(\mathbb{T}))^{-1}$.  So by Proposition \ref{saras}, \[\phi=\chi_n\exp(w-i\widetilde{w})\exp(ig)\] for some $n\in\mathbb{Z}$, $w\in L_\mathbb{R}^\infty$ and $g\in VMO_{\mathbb{R}}$.  This yields \[\exp(2w)=\phi\overline{\phi}\in \mathcal{I}.\]  This implies $0\notin \sigma\left(\exp(2w)\right)$ where $\sigma\left(\exp(2w)\right)$ is the spectrum of $\exp(2w)$ as an element of $QC$.  Then by Corollary 9.5 \cite[page 57]{zhuAl} and the fact that $\sigma\left(\exp(2w)\right)$ is compact, $\sigma\left(\exp(2w)\right)\subseteq [c,\|\exp(2w)\|_\infty]$ for some $c>0$.  Then by the continuous functional calculus \cite[page 62]{zhuAl}, $\frac{1}{2}\ln(\exp(2w))\in QC$.  It follows that $w\in QC_\mathbb{R}$.  Then by \eqref{quasi} and \eqref{doublevmo}, $\widetilde{w}\in VMO_{\mathbb{R}}$.  This means \[\phi=\chi_n\exp(w)\exp(ih)\] where $h=\widetilde{w}+g$.  Hence \[\mathcal{I}\subseteq \{\chi_n\}_n\exp\left(QC_\mathbb{R}\right)\exp\left(iVMO_{\mathbb{R}}\right).\]  Therefore Theorem \ref{classI} holds.
\end{proof}

\subsection{Corollary to Theorem \ref{classI}}

For any $k\in\mathbb{Z}$, let \[I_k=\{\phi\in \mathcal{I}:\ind(\phi)=k\}.\]  Then by Theorem \ref{classI} and the next lemma, we can completely classify $I_k$.

\begin{lem}\label{ind}
For any $g\in VMO_{\mathbb{R}}$, $w\in L_\mathbb{R}^\infty$ and $n\in\mathbb{Z}$, \[\ind(\chi_n)=n, \ \ind(\exp(ig))=0\text{ and } \ind(\exp(w-i\widetilde{w}))=0.\]

\end{lem}

\begin{proof}

Let $g\in VMO_{\mathbb{R}}$, $w\in L_\mathbb{R}^\infty$ and $n\in\mathbb{Z}$.  By Lemma 6.4.3 \cite[page 151-152]{Doug}, $\widehat{1}=1$ and $\widehat{\chi_n}(z)=z^n$ for all $z\in\mathbb{D}$ and all $n\in\mathbb{Z}$.  This yields \begin{equation}\label{indexConst}\ind(\chi_n)=n \text{ for every }n\in\mathbb{Z}.\end{equation}

Since $w\in L_\mathbb{R}^\infty$, $\exp(w-i\widetilde{w})$ is invertible in $H^\infty$ by Lemma \ref{inHardy}.  Moreover with $f=\exp(w-i\widetilde{w})$, $T_f\in \mathfrak{F}(H^2)$ by Lemma \ref{fredI} below and $T_f$ is invertible by Proposition \ref{commToe}.  It follows by Corollary 7.25 \cite[page 165]{Doug} that $j(T_f)=0$.  Hence by \eqref{index}, \begin{equation}\label{indL}\ind(\exp(w-i\widetilde{w}))=0 \text{ for every } w\in L_\mathbb{R}^\infty.\end{equation}

By \eqref{vmoReal}, $\exp(ig)=\exp(i(u-\widetilde{v}))=\exp(iu)\exp(-i\widetilde{v})$ for some $u,v\in C_\mathbb{R}(\mathbb{T})$.  Now  by Proposition \ref{saras} and the fact that $L^\infty$ is a commutative Banach Algebra with respect to function multiplication, the function $F:[0,1]\rightarrow (H^\infty+C(\mathbb{T}))^{-1}$ defined by $F(\lambda)=\exp(-\lambda v)$ is a path in $(H^\infty+C(\mathbb{T}))^{-1}$ between 1 and $\exp(-v)$.  It follows that $\ind(\exp(-v))=\ind(1)=0$.  Then since $\exp(-i\widetilde{v})=\exp(v-i\widetilde{v})\exp(-v)$, \eqref{indL} implies $\ind(\exp(-i\widetilde{v}))=0$.  Similarly the function $G:[0,1]\rightarrow (H^\infty+C(\mathbb{T}))^{-1}$ defined by $G(\lambda)=\exp(i\lambda u)$ is a path in $(H^\infty+C(\mathbb{T}))^{-1}$ between $\exp(iu)$ and 1.  Hence $\ind(\exp(iu))=0$.  So by Lemma \ref{indSum} \begin{equation}\label{indexVMO}\ind(\exp(ig))=0.\end{equation}  Therefore $\ind(\chi_n)=n$ and $\ind(\exp(ig))=\ind(\exp(w-i\widetilde{w}))=0.$
\end{proof}

\begin{cor}\label{classIk}
\begin{equation}\label{kComp}\begin{split}
& I_0 = \exp\left(QC_\mathbb{R}\right)\exp\left(iVMO_{\mathbb{R}}\right) \\& \text{and } I_k=\chi_kI_0 \ \text{ for all } k\in\mathbb{Z}\end{split}\end{equation}
\end{cor}
\begin{proof}
Let $\chi_n\exp(ig)\exp(w)\in\mathcal{I}$ where $g\in VMO_\mathbb{R}$, $w\in QC_\mathbb{R}$ and $n\in\mathbb{Z}$.  Clearly $\exp(ig)\exp(w)=\exp(i(g+\widetilde{w}))\exp(w-i\widetilde{w})$.  Then for any $k\in\mathbb{Z}$, Lemma \ref{ind} yields \[\ind(\chi_n\exp(ig)\exp(w))=k \text{ if and only if } n=k.\]  Therefore for any $n,k\in\mathbb{Z}$, $\chi_n\exp(ig)\exp(w)\in I_k$ if and only if $n=k$ and \eqref{kComp} holds.
\end{proof}

It is implied from \eqref{kComp} that in order to fully understand the sets $\{I_k\}_{k\in\mathbb{Z}}$, it suffices to study $I_0$.  This is the approach we will use to classify the path-connected components of $\mathcal{I}$.


\section{Path-connected components}

We will also need the following lemma in order to classify the path-connected components of $\mathcal{I}$ and of $\mathcal{F}$. The lemma follows from the definition of the path-connected components of a topological space \cite[page 160]{MunTop}.

\begin{lem}\label{equiv}
Let $X$ be a topological space and $\{A_\alpha\}_{\alpha\in \Omega}$ be a collection of subspaces of $X$.  Then $\{A_\alpha\}_{\alpha\in \Omega}$ are the path-connected components of $X$ if all of the following conditions hold:
\begin{enumerate}
\item $X=\bigcup_{\alpha\in \Omega}A_\alpha$.

\item $\{A_\alpha\}_{\alpha\in \Omega}$ are pairwise disjoint.

\item $\{A_\alpha\}_{\alpha\in \Omega}$ are all path-connected in $X$.

\item Any nonempty path-connected subspace of $X$ intersects only one $A_\alpha$ from $\{A_\alpha\}_{\alpha\in \Omega}$.

\end{enumerate}

\end{lem}

\subsection{Special Case $I_0$}

\subsubsection{First the Unimodular functions}

Let $U_0$ be the set of all unimodular functions in $I_0$.  Then by Corollary \ref{classIk}, \begin{equation}\label{unimDef}U_0=\exp\left(iVMO_{\mathbb{R}}\right).\end{equation} We will first prove $\{P_g\}_{\left\{\begin{subarray}{l} g\in VMO_{\mathbb{R}}\setminus L_\mathbb{R}^\infty \text{ or }\\ g=0\end{subarray}\right\}}$ are the path-connected components of $U_0$ where for any $g\in VMO_\mathbb{R}$, $P_g$ is defined by \eqref{pPath}.

\begin{lem}\label{pathCond}
Let $\exp(ig), \exp(ih) \in U_0$ where $g,h\in VMO_{\mathbb{R}}$.  Then $\exp(ig)$ is path-connected to $\exp(ih)$ in $U_0$ $\text{ if and only if } g-h\in QC_{\mathbb{R}}$.
\end{lem}

\begin{proof}
Let $\exp(ig)$ and $\exp(ih)$ be as above.  Clearly $\exp(ig)$ is path-connected to $\exp(ih)$ in $U_0$ if and only if $\exp(i(g-h))$ is path-connected to 1 in $U_0$.  So it suffices to show $\exp(i(g-h))$ is path-connected to 1 in $U_0$ if and only if $g-h\in QC_{\mathbb{R}}$.

$(\Rightarrow)$  Assume $\exp(i(g-h))$ is path-connected to 1 in $U_0$.  Recall $G_0$ is the connected component of $\mathcal{I}$ containing 1.  Then $\exp(i(g-h))\in G_0$ and Lemma \ref{exp} implies \[\exp(i(g-h))=\exp(u+iv)\] for some $u+iv\in QC$ where $u$ and $v$ are real-valued.  It follows that \[1=\exp(u)\mbox{ and } \exp(i(g-h))=\exp(iv).\]  Thus we have \begin{equation}\label{uniEq}u=0 \mbox{ and } g-h=v+2\pi\sigma \text{ for some } \sigma:\mathbb{T}\rightarrow\mathbb{Z}.\end{equation}  Since $u+iv\in QC$, direct calculation and \eqref{quasi} imply $v\in QC_{\mathbb{R}}$.  Hence by \eqref{uniEq}, $\sigma\in VMO$.  Then by Corollary \ref{Z}, $\sigma$ is constant.  Therefore $g-h\in QC_{\mathbb{R}}$.

$(\Leftarrow)$  Now assume $g-h\in QC_{\mathbb{R}}$.  Since $QC$ is a commutative Banach Algebra, the function $G:[0,1]\rightarrow U_0$ defined by $G(\lambda)=\exp(i\lambda(g-h))$ is continuous on $[0,1]$.  It follows that $G$ is a path between $\exp(i(g-h))$ and 1 in $U_0$.  Therefore $\exp(i(g-h))$ is path-connected to 1 in $U_0$.

\end{proof}

\begin{lem}\label{pathComp} The path-connected components of $U_0$ are \\$\{P_g\}_{\left\{\begin{subarray}{l} g\in VMO_{\mathbb{R}}\setminus L_\mathbb{R}^\infty \text{ or }\\ g=0\end{subarray}\right\}}$
\end{lem}

\begin{proof}

From Lemma \ref{pathCond}, \[P_g \text{ is path-connected in }U_0\text{ for all } g\in VMO_{\mathbb{R}}.\]  Let $\exp(ih)\in P_w\cap P_q$ with $h\in VMO_{\mathbb{R}}$, $P_w,P_q\in \{P_g\}_{\left\{\begin{subarray}{l} g\in VMO_{\mathbb{R}}\setminus L_\mathbb{R}^\infty \text{ or }\\ g=0\end{subarray}\right\}}$.  Then by \eqref{pPath}, $h-w, h-q\in QC_{\mathbb{R}}$.  Hence $w-q\in QC_{\mathbb{R}}$.  This means $P_w=P_q$.  It follows that $\{P_g\}_{\left\{\begin{subarray}{l} g\in VMO_{\mathbb{R}}\setminus L_\mathbb{R}^\infty \text{ or }\\ g=0\end{subarray}\right\}}$ is pairwise disjoint.


By \eqref{pPath} and \eqref{unimDef} $P_g\subseteq U_0 \ \text{ for all } g \in VMO_{\mathbb{R}}$.  This yields $P_0\cup \bigcup_{g\in VMO_{\mathbb{R}}\setminus L_\mathbb{R}^\infty} P_g\subseteq U_0$.  Let $\exp(ih)\in U_0$.  Then by definition of $U_0$, $h\in VMO_\mathbb{R}$.  If $h\in L_\mathbb{R}^\infty$, then by \eqref{quasi}, $h\in QC_\mathbb{R}$.  Thus $\exp(ih)\in P_0$.  If $h\notin L_\mathbb{R}^\infty$, then $\exp(ih)\in P_h\subseteq\bigcup_{g\in VMO_{\mathbb{R}}\setminus L_\mathbb{R}^\infty} P_g $.  So $U_0\subseteq P_0\cup\bigcup_{g\in VMO_{\mathbb{R}}\setminus L_\mathbb{R}^\infty} P_g$.  Thus \[U_0=P_0\cup \bigcup_{g\in VMO_{\mathbb{R}}\setminus L_\mathbb{R}^\infty} P_g.\]

Let $V$ be any nonempty path-connected subspace of $U_0$ such that $V\cap P_h\neq\emptyset$ and $V\cap P_w\neq\emptyset$ for some $P_w,P_h\in\{P_g\}_{\left\{\begin{subarray}{l} g\in VMO_{\mathbb{R}}\setminus L_\mathbb{R}^\infty \text{ or }\\ g=0\end{subarray}\right\}}$.  Then since $\exp(ih)\in P_h$ and $\exp(iw)\in P_w$, $\exp(ih)$ is path-connected to $\exp(iw)$ in $U_0$.  So by Lemma \ref{pathCond} $w-h\in QC_\mathbb{R}$ and $P_h=P_w$.  Therefore by Lemma \ref{equiv} $\{P_g\}_{\left\{\begin{subarray}{l} g\in VMO_{\mathbb{R}}\setminus L_\mathbb{R}^\infty \text{ or }\\ g=0\end{subarray}\right\}}$ are the path-connected components of $U_0$.

\end{proof}

\subsubsection{Path-connected components in $I_0$}

We can now completely classify the path-connected components of $I_0$.

\begin{thm}\label{connIo}
The path-connected components of $I_0$ are \newline$\{Q_g\}_{\left\{\begin{subarray}{l} g\in VMO_{\mathbb{R}}\setminus L_\mathbb{R}^\infty \text{ or }\\ g=0\end{subarray}\right\}}$ where $Q_g$ is defined by \eqref{uniPath} for any $g\in VMO_\mathbb{R}$.

\end{thm}
\begin{proof}
Let $x,y\in QC_\mathbb{R}$.  Since $QC$ is a Banach Algebra, the function $F:[0,1]\rightarrow \exp(QC_\mathbb{R})$ defined by $F(\lambda)=\exp(x+\lambda(y-x))$ is a path between $\exp(x)$ and $\exp(y)$.  Thus $\exp(QC_\mathbb{R})$ is a path-connected subset of $\mathcal{I}$.  So by Lemma \ref{pathComp} and Corollary \ref{classIk}, \[Q_g \text{ is path-connected in } I_0\text{ for all } g\in VMO_{\mathbb{R}}.\]

Let $\exp(w)\exp(ih)\in Q_q\cap Q_y$, where $w\in QC_{\mathbb{R}}$, $h\in VMO_\mathbb{R}$, and $Q_q,Q_y\in \{Q_g\}_{\left\{\begin{subarray}{l} g\in VMO_{\mathbb{R}}\setminus L_\mathbb{R}^\infty \text{ or }\\ g=0\end{subarray}\right\}}$.  Then for some $u\in QC_\mathbb{R}$ and $v\in VMO_\mathbb{R}$ with $\exp(iv)\in P_q$, $\exp(w)\exp(ih)=\exp(u)\exp(iv)$.  Then $w=u$ and $h-v=2\pi\eta$ for some integer valued function $\eta$ on $\mathbb{T}$.  By Corollary \ref{Z}, $\eta$ is constant and $\exp(ih)\in P_q$.  Similarly $\exp(ih)\in P_y$.  Then by Lemma \ref{pathComp}, $P_q=P_y$.  This means $\{Q_g\}_{\left\{\begin{subarray}{l} g\in VMO_{\mathbb{R}}\setminus L_\mathbb{R}^\infty \text{ or }\\ g=0\end{subarray}\right\}}$ is pairwise disjoint.


By Corollary \ref{classIk} and Lemma \ref{pathComp}, \[I_0=\exp(QC_\mathbb{R})U_0=\exp(QC_\mathbb{R})P_0\cup \bigcup_{g\in VMO_{\mathbb{R}}\setminus L_\mathbb{R}^\infty} \exp(QC_\mathbb{R})P_g.\]  Hence by \eqref{uniPath}, \[I_0=Q_0\cup\bigcup_{g\in VMO_{\mathbb{R}}\setminus L_\mathbb{R}^\infty}Q_g.\]

Let $V$ be any nonempty path-connected subspace of $I_0$ satisfying $V\cap Q_h\neq\emptyset$ and $V\cap Q_w\neq\emptyset$ for some $Q_h,Q_w\in\{Q_g\}_{\left\{\begin{subarray}{l} g\in VMO_{\mathbb{R}}\setminus L_\mathbb{R}^\infty \text{ or }\\ g=0\end{subarray}\right\}}$.  Then since $\exp(ih)\in Q_h$ and $\exp(iw)\in Q_w$, $\exp(ih)$ is path-connected to $\exp(iw)$ in $I_0$ by some path $F$.  This means $\frac{F}{|F|}$ is a path in $U_0$ between $\exp(iw)$ and $\exp(ih)$. So by Lemma \ref{pathComp}, $P_h= P_w$.  Thus $Q_h=Q_w$.  Therefore by Lemma \ref{equiv} $\{Q_g\}_{\left\{\begin{subarray}{l} g\in VMO_{\mathbb{R}}\setminus L_\mathbb{R}^\infty \text{ or }\\ g=0\end{subarray}\right\}}$ are the path-connected components of $I_0$.

\end{proof}

\subsection{Case of $I_k$ with $k\neq0$}

Now that we have classified the path-connected components of $I_0$, we can classify the path-connected components of $I_k$ for any $k\in\mathbb{Z}\setminus\{0\}$.  This part of the classification is very easy due to the classification of $I_0$ and Proposition \ref{homIndex} below.

\begin{prop} \label{homIndex}
For any $k\in\mathbb{Z}$, $I_k$ is homeomorphic to $I_0$.
\end{prop}

\begin{proof}
Let $k\in\mathbb{Z}$.  Let $\Gamma_k:I_0\rightarrow I_k$ and $\Theta_k:I_k\rightarrow I_0$ be defined by \[\Gamma_k(\phi)=\chi_k\phi \mbox{ and } \Theta_k(\psi)=\chi_{-k}\psi.\]  Both $\Gamma_k$ and $\Theta_k$ are well-defined by Corollary \ref{classIk}.  Since $QC$ is a Banach Algebra, $\Gamma_k$ and $\Theta_k$ are both continuous on their domains.  Moreover since $\left(\chi_k\right)^{-1}=\chi_{-k}$ with respect to multiplication, we have $\Gamma_k^{-1}=\Theta_k$.  Therefore $\Gamma_k$ is a homeomorphism and $I_k$ is homeomorphic to $I_0$.
\end{proof}

\begin{thm}\label{connIk}
For any $k\in\mathbb{Z}$, the path-connected components of $I_k$ are $\{\chi_kQ_g\}_{\left\{\begin{subarray}{l}g\in VMO_{\mathbb{R}}\setminus L_\mathbb{R}^\infty \text{ or }\\ g=0\end{subarray}\right\}}$.
\end{thm}
\begin{proof}

Let $\Gamma_k$ be as in Proposition \ref{homIndex}.  Then \[\Gamma_k(Q_g)=\chi_kQ_g \text{ for all } g\in VMO_{\mathbb{R}}.\]  Therefore a straightforward calculation using $\Gamma_k$ and $\Theta_k$ along with Theorem \ref{connIo} , Proposition \ref{homIndex} and Lemma \ref{equiv}, $\{\chi_kQ_g\}_{\left\{\begin{subarray}{l} g\in VMO_{\mathbb{R}}\setminus L_\mathbb{R}^\infty \text{ or }\\ g=0\end{subarray}\right\}}$ are the path-connected components of $I_k$.
\end{proof}

\section{Proof of Theorem \ref{mainI}}

\begin{proof}  By Theorem \ref{connIk}, \[\{\chi_kQ_g\}_{\left\{\begin{subarray}{l} k\in\mathbb{Z} \text{ and either }\\ g\in VMO_{\mathbb{R}}\setminus L_\mathbb{R}^\infty \text{ or }\\ g=0\end{subarray}\right\}} \text{ is a collection of path-connected sets in } \mathcal{I}\] and \[\mathcal{I}=\bigcup_{\left\{\begin{subarray}{l} k\in\mathbb{Z} \text{ and either } \\ g\in VMO_{\mathbb{R}}\setminus L_\mathbb{R}^\infty \text{ or }g=0\end{subarray}\right\}}\chi_kQ_g.\]

Let $\chi_p\exp(w)\exp(ih)\in \chi_nQ_q\cap \chi_mQ_y$ for some $w\in QC_\mathbb{R}$, $h\in VMO_\mathbb{R}$, $p\in\mathbb{Z}$ and $\chi_mQ_y,\chi_nQ_q\in\{\chi_kQ_g\}_{\left\{\begin{subarray}{l} k\in\mathbb{Z} \text{ and either }\\ g\in VMO_{\mathbb{R}}\setminus L_\mathbb{R}^\infty \text{ or }\\ g=0\end{subarray}\right\}}$.  Then by by Theorem \ref{connIk}, $\chi_p\exp(w)\exp(ih)\in I_n\cap I_m$.  It follows that $\ind(\chi_p\exp(w)\exp(ih))=m$ and $\ind(\chi_p\exp(w)\exp(ih))=n$.  So by Lemma \ref{ind} $m=n=p$ and by Theorem \ref{connIk}, $\chi_nQ_g=\chi_mQ_y$.  That is $\{\chi_kQ_g\}_{\left\{\begin{subarray}{l} k\in\mathbb{Z} \text{ and either }\\ g\in VMO_{\mathbb{R}}\setminus L_\mathbb{R}^\infty \text{ or }\\ g=0\end{subarray}\right\}}$ is pairwise disjoint.

Let $W$ be a nonempty path-connected subspace of $\mathcal{I}$ so that $W\cap \chi_nQ_q\neq\emptyset$ and $W\cap \chi_mQ_y\neq\emptyset$ for some $\chi_mQ_y,\chi_nQ_q\in\{\chi_kQ_g\}_{\left\{\begin{subarray}{l} k\in\mathbb{Z} \text{ and either }\\ g\in VMO_{\mathbb{R}}\setminus L_\mathbb{R}^\infty \text{ or }\\ g=0\end{subarray}\right\}}$.  Then since $\chi_n\exp(iq)\in \chi_nQ_q$ and $\chi_m\exp(iy)\in \chi_mQ_y$, $\chi_n\exp(iq)$ is path-connected to $\chi_m\exp(iy)$ in $\mathcal{I}$.  Let $F:[0,1]\rightarrow\mathcal{I}$ be the path.  Since $\ind$ is continuous on $\mathcal{I}$, we must have $m=n=\ind(F(\lambda))$ for all $\lambda\in[0,1]$.  Hence $\chi_n\exp(iq)$ is path-connected to $\chi_n\exp(iy)$ by $F$ in $I_n$.  Then from Theorem \ref{connIk}, $Q_y=Q_q$.  Therefore $\chi_mQ_y=\chi_nQ_q$ and by Lemma \ref{equiv}, $\{\chi_kQ_g\}_{\left\{\begin{subarray}{l} k\in\mathbb{Z} \text{ and either }\\ g\in VMO_{\mathbb{R}}\setminus L_\mathbb{R}^\infty \text{ or }\\ g=0\end{subarray}\right\}}$ are the path-connected components of $\mathcal{I}$.

\end{proof}

\section{Uncountably many path-connected components of $\mathcal{I}$}\label{uncount}
In this section we will prove $\mathcal{I}$ has uncountably many path-connected components.  To do this we will construct an explicit example of a function $H$ such that $H\in VMO_\mathbb{R}\setminus L_\mathbb{R}^\infty$.  We will also use $H$ later on to show $\mathcal{F}$ has uncountably many path-connected components.  To construct $H$, we will need two following theorems from  \cite[page 182]{TrigSer} and \cite[page 101]{SeriesFourier} respectively .

\begin{thm}\label{trigThm}
Suppose $a_v\geq a_{v+1}$ and $a_v\rightarrow0$.  Then a necessary and sufficient condition for the uniform convergence of $\sum_{v=1}^\infty a_v\sin(vx)$ is $va_v\rightarrow 0$.
\end{thm}

\begin{thm}\label{fourSeri}
If the coefficients $a_n$ and $b_n$ of the series \[\frac{a_0}{2}+\sum_{n=1}^\infty a_n\cos(nx), \sum_{n=1}^\infty b_n\sin(nx)\] are positive and decrease monotonically to zero as $n\rightarrow\infty$, then both series converge uniformly on any interval $[a,b]$ which does not contain points of the form $x=2k\pi$ $(k=0,\pm 1\pm 2,\ldots)$.
\end{thm}

Now we will construct $H$.

\begin{prop}\label{Ex}
There exists a function $H\in VMO_\mathbb{R}$ such that $H\notin L_\mathbb{R}^\infty$.
\end{prop}

\begin{proof}

For each $N\geq 2$ let $g_N:\mathbb{R}\rightarrow\mathbb{R}$ be defined by $g_N(x)=\sum_{k=2}^N \frac{\sin(kx)}{k\ln(k)}$.  By Theorem \ref{trigThm}, $\{g_N\}_{N=2}^\infty$ converges uniformly on $\mathbb{R}$.  Let $g:\mathbb{R}\rightarrow \mathbb{R}$ be defined by $g(x) =\sum_{k=2}^\infty \frac{\sin(kx)}{k\ln(k)}$.  Then $g\in C_\mathbb{R}(\mathbb{T})$ and by \eqref{vmoReal}, $\widetilde{g}\in VMO_\mathbb{R}$.

Direct calculation shows $\widetilde{g_N}(x)=-\sum_{k=2}^N \frac{\cos(kx)}{k\ln(k)}$ for all $x\in\mathbb{R}$.  Recall that by the proof of Lemma \ref{DblHil}, the Hilbert transform is bounded on $L^2(\mathbb{T})$ with respect to the $L^2(\mathbb{T})$ norm $\|\cdot\|_2$.  So by the preceding paragraph, $\lim_{N\rightarrow\infty}\|\widetilde{g_N}-\widetilde{g}\|_2 =0$.  Also by Theorem \ref{fourSeri}, $-\sum_{k=2}^\infty \frac{\cos(kx)}{k\ln(k)}$ converges uniformly on each closed subinterval of $\mathbb{R}\setminus \{2k\pi\}_{k\in\mathbb{Z}}$.  Hence $\widetilde{g}(x)=-\sum_{k=2}^\infty \frac{\cos(kx)}{k\ln(k)}$ for almost every $x\in\mathbb{R}$.  Note that since $\sum_{k=2}^\infty \frac{1}{k\ln(k)}=\infty$, $\{\widetilde{g_N}(0)\}_{N=2}^\infty$ is not bounded.  This means $\{\widetilde{g_N}\}_{n=2}^\infty$ is not uniformly bounded in $N$ and $x$.

Let $H:\mathbb{R}\rightarrow\mathbb{R}$ be defined by \[H(x) = \left\{
        \begin{array}{ll}
            -\sum_{k=2}^\infty \frac{\cos(kx)}{k\ln(k)}& \mbox{ if }\ \widetilde{g}(x)=-\sum_{k=2}^\infty \frac{\cos(kx)}{k\ln(k)}\\
            0 & \mbox{otherwise}
        \end{array}
    \right.\]

Then $H\in VMO_\mathbb{R}$.  Suppose $H\in L_\mathbb{R}^\infty$.  Let $\sigma_N:\mathbb{T}\rightarrow \mathbb{C}$ be defined by $\sigma_N(x)=\frac{1}{2(N+1)\pi}\int_0^{2\pi} H(x-t)\left(\frac{\sin\left(\frac{(N+1)t}{2}\right)}{\sin\left(\frac{t}{2}\right)}\right)^2dt$.  Then $\left|\sigma_N(x)\right|\leq \|H\|_\infty$ for every $N\geq 2$ and $x\in\mathbb{T}$ \cite[page 63]{krantz}.  Moreover by \cite[page 60]{krantz} $\sigma_N(x)=\frac{1}{n+1}\sum_{j=2}^n \widetilde{g_N}(x)$. Furthermore a straightforward induction argument shows $\widetilde{g_N}(x)-\sigma_N(x)=\frac{1}{N+1}\sum_{j=2}^N \frac{\cos(jx)}{\ln(j)}$.  Hence for some $M_1>0$, $|\widetilde{g_N}(x)-\sigma_N(x)|\leq M_1$ for all $N\geq 2$ and all $x\in\mathbb{T}$.  It follows that $\{\widetilde{g_N}(x)\}_{N=2}^\infty$ is uniformly bounded in $x$ and $N$, which is a contradiction by the above statements.  So $H\notin L_\mathbb{R}^\infty$.  Therefore $H\in VMO_\mathbb{R}\setminus L_\mathbb{R}^\infty$.

\end{proof}

\begin{thm}\label{uncountI}
$\mathcal{I}$ has uncountably many path-connected components.
\end{thm}

\begin{proof}
Let $H$ be as in Proposition \ref{Ex}.  Then by \eqref{quasi}, $H\notin QC_{\mathbb{R}}$.  By \eqref{pPath}, \eqref{uniPath} and Corollary \ref{Z}, \[Q_g=Q_q \text{ if and only if } g-q\in QC_{\mathbb{R}}\] for any $g,q\in VMO_{\mathbb{R}}$.  Hence for all $\beta, \sigma\in\mathbb{R}$, \[\beta \neq\sigma \text{ implies }Q_{\beta H}\neq Q_{\sigma H}.\]  Thus $\{Q_{\beta H}\}_{\beta\in\mathbb{R}}$ is an uncountable collection of path-connected components in $I_0$ and $I_0$ has uncountably many path-connected components.  Therefore by Theorem \ref{mainI}, $\mathcal{I}$ has uncountably many path-connected components.
\end{proof}

\begin{rem}
In fact by Theorem \ref{connIk} and Theorem \ref{uncountI}, $I_k$ has uncountable many path-connected components for all $k\in\mathbb{Z}$.

\end{rem}

\section{Classification of $\mathcal{F}$ and the Path-connected components of $\mathcal{F}$} \label{fred}
Here we will classify $\mathcal{F}$ and the path-connected components of $\mathcal{F}$.  As we will see, the classification of the path-connected components of $\mathcal{F}$ is based on the classification of the path-connected components of $\mathcal{I}$.  So one can really say that Theorem \ref{mainF} is a corollary to Theorem \ref{mainI}.  We will also show that $\mathcal{F}$ has uncountably many path-connected components.

For this section, we will use the following lemma which holds by proposition 7.12 \cite[page 161]{Doug} and the proof of proposition 7.11 \cite[page 161]{Doug}.

\begin{lem}\label{compactToe}
Let $f\in L^\infty(\mathbb{T})$ and $K\in\mathfrak{LC}(H^2)$.  Then $\|T_f+K\|\geq\|T_f\|$.
\end{lem}

The above implies that the only compact Toeplitz operator on $H^2$ is the zero operator.  That is, $T_f\in\mathfrak{LC}(H^2)$ if and only if $f=0$.

\section{Classification of $\mathcal{F}$}

Recall from above that $\mathfrak{I}(QC)$ is the $C^*$-subalgebra of $\mathfrak{L}(H^2)$ generated by $\{T_f\}_{f\in QC}$.  The next lemma is well-known in the literature.

\begin{lem}\label{ToeQC}
$\mathfrak{I}(QC)=\{T_f+K:f\in QC, K\in\mathfrak{LC}(H^2)\}$
\end{lem}
We will also need the following lemma, which is Corollary 7.34 \cite[page 168]{Doug}.  \begin{lem}\label{fredI}
Let $f\in H^\infty+C(\mathbb{T})$.  Then $T_f\in\mathfrak{F}(H^2)$ if and only if $f\in (H^\infty+C(\mathbb{T}))^{-1}$.
\end{lem}

We can now prove the following Theorem which describes $\mathcal{F}$ completely.

\begin{thm}\label{F}
$\mathcal{F}=\{T_f+K:f\in \mathcal{I}, K\in \mathfrak{LC}(H^2)\}$.
\end{thm}

\begin{proof}
By Lemma \ref{ToeQC}, $\mathcal{F}=\mathfrak{F}(H^2)\cap\{T_f+K:f\in QC, K\in\mathfrak{LC}(H^2)\}$.  Also by proposition 5.15 \cite[page 113]{Doug}, $T_f+K\in\mathfrak{F}(H^2)$ if and only if $T_f\in\mathfrak{F}(H^2)$ for any $K\in\mathfrak{LC}(H^2)$ and any $f\in L^\infty\left(\mathbb{T}\right)$.  So it suffices to show $T_f\in\mathcal{F}$ if and only if $f\in \mathcal{I}$.

$(\Leftarrow)$ If $f\in \mathcal{I}$ then $f\in (H^\infty+C(\mathbb{T}))^{-1}$.  So by Lemma \ref{fredI}, $T_f\in \mathfrak{F}(H^2)$.  Then by definition of $\mathcal{F}$ and Lemma \ref{ToeQC}, $T_f\in \mathcal{F}$.

$(\Rightarrow)$ Suppose $T_f\in \mathcal{F}$.  Then $T_f\in\mathfrak{F}(H^2)$ and $T_f=T_g+K$ for some $g\in QC$ and $K\in\mathfrak{LC}(H^2)$.  Then $T_{f-g}\in\mathfrak{LC}(H^2)$ which means $f=g$ by Lemma \ref{compactToe}.  So $f\in QC$.  Also, by proposition 5.15 \cite[page 115]{Doug}, $T_{\overline{f}}\in\mathfrak{F}(H^2)$.  Then by Lemma \ref{fredI}, $f, \overline{f}\in (H^\infty+C(\mathbb{T}))^{-1}$.  Therefore by Lemma \ref{I}, $f\in \mathcal{I}$.
\end{proof}

\section{Path-connected components of the elements of $\mathcal{F}$ of index k}
For any $k\in\mathbb{Z}$, let \[E_k=\{T_f+K\in \mathcal{F}: j(T_f+K)=k\}.\]  By Section \ref{SectInd}, \[j(T_f+K)=j(T_f)=-\ind(f) \text{ for all } f\in \mathcal{I} \text{ and for all } K\in\mathfrak{LC}(H^2).\]  Hence \begin{equation}\label{firstDefEk}E_k=\{T_f+K\in \mathcal{F}: f\in I_{-k}, K\in\mathfrak{LC}(H^2)\}.\end{equation}

Let $\zeta:L^\infty\left(\mathbb{T}\right)\rightarrow \mathfrak{L}(H^2)$ be the map defined by $\zeta(f)=T_f$.  Clearly $\zeta$ is well-defined.  Moreover by proposition 7.4 \cite[page 159]{Doug} and Corollary 7.8 \cite[page 160]{Doug}, $\zeta$ is *-linear and an isometry.  Let $\xi=\zeta|_{\mathcal{I}}$.  Notice $\xi$ is continuous on $\mathcal{I}$.  We will now use $\xi$ and Theorem \ref{connIk} to classify the path-connected components of $E_k$.  Recall for each $g\in VMO_{\mathbb{R}}$ and $k\in\mathbb{Z}$, ${_k}V_g$ is defined by \eqref{fredPath}.

\begin{thm}\label{pathEk}
The path-connected components of $E_k$ are \[\{_kV_g\}_{\left\{\begin{subarray}{l} g\in VMO_{\mathbb{R}}\setminus L_\mathbb{R}^\infty \text{ or } \\ g=0\end{subarray}\right\}}.\]
\end{thm}

\begin{proof}
 As mentioned in the proof of Lemma \ref{ToeQC}, $\mathfrak{LC}(H^2)$ is a two-sided *-ideal in $\mathfrak{L}(H^2)$.  This implies $\mathfrak{LC}(H^2)$ is convex in $\mathfrak{L}(H^2)$.  Also for any $g\in VMO_{\mathbb{R}}$ and any $k\in\mathbb{Z}$, \begin{equation}\label{xi}{_k}V_g=\xi(\chi_{-k}Q_g)+\mathfrak{LC}(H^2).\end{equation}  It then follows from Theorem \ref{connIk} that \[{_k}V_g \text{ is path-connected in } E_k \text{ for all } g\in VMO_{\mathbb{R}}.\]  By Theorem \ref{connIk} and \eqref{fredPath}, we can see that \[E_k = \bigcup_{\left\{\begin{subarray}{l} g\in VMO_{\mathbb{R}}\setminus L_\mathbb{R}^\infty \text{ or } \\ g=0\end{subarray}\right\}} {_k}V_g. \]

Let $T_{f_0}+K_0\in {_k}V_q\cap {_k}V_h$ for some $f_0\in I_{-k}$, $K_0\in\mathfrak{LC}(H^2)$ and ${_k}V_q, {_k}V_h\in\{_kV_g\}_{\left\{\begin{subarray}{l} g\in VMO_{\mathbb{R}}\setminus L_\mathbb{R}^\infty \text{ or } \\ g=0\end{subarray}\right\}}$. Then for some $f_1\in\chi_{-k}Q_h$ and $K_1\in\mathfrak{LC}(H^2)$, $T_{f_0}+K_0=T_{f_1}+K_1$.  Thus by Lemma \ref{compactToe}, $f_0=f_1$ and $f_0\in\chi_{-k}Q_h$.  Similarly, $f_0\in\chi_{-k}Q_q$.  Then by Theorem \ref{connIk}, ${_k}V_q= {_k}V_h$.  Hence $\{_kV_g\}_{\left\{\begin{subarray}{l} g\in VMO_{\mathbb{R}}\setminus L_\mathbb{R}^\infty \text{ or } \\ g=0\end{subarray}\right\}}$ is pairwise disjoint.

Let $W$ be a nonempty path-connected subspace of $E_k$ with ${_k}V_h\cap W\neq\emptyset$ and ${_k}V_q\cap W\neq\emptyset$ for some ${_k}V_h,{_k}V_q\in\{_kV_g\}_{\left\{\begin{subarray}{l} g\in VMO_{\mathbb{R}}\setminus L_\mathbb{R}^\infty \text{ or } \\ g=0\end{subarray}\right\}}$.  Let $f_h=\chi_{-k}\exp(ih)$ and $f_q=\chi_{-k}\exp(iq)$.  Then since $T_{f_h}\in {_k}V_h$ and $T_{f_q}\in {_k}V_q$, $T_{f_h}$ is path-connected to $T_{f_q}$ in $E_k$.  Let $H:[0,1]\rightarrow E_k$ be the path between $T_{f_q}$ and $T_{f_h}$.  For each $\lambda\in[0,1]$, let $T_{f_\lambda}+K_\lambda=H(\lambda)$ where $f_0=f_q$, $f_1=f_h$, $K_\lambda\in \mathfrak{LC}(H^2)$ for each $\lambda\in (0,1)$ and $K_0,K_1$ both equal the zero operator.  Then by Lemma \ref{compactToe}, \[\|f_{\lambda}\|_\infty=\|T_{f_\lambda}\|\leq\|H(\lambda)\| \text{ for all } \lambda\in [0,1].\]  It follows that $f_q$ is path-connected to $f_h$ in $I_{-k}$.  Then from Theorem \ref{connIk}, $\chi_{-k}Q_h=\chi_{-k}Q_q$.  Hence ${_k}V_q= {_k}V_h$ and by Lemma \ref{equiv}, \newline$\{_kV_g\}_{\left\{\begin{subarray}{l} g\in VMO_{\mathbb{R}}\setminus L_\mathbb{R}^\infty \text{ or } \\ g=0\end{subarray}\right\}}$ are the path-connected components of $E_k$.

\end{proof}
\section{Proof of Theorem \ref{mainF}}

\begin{proof}[Proof of Theorem \ref{mainF}]

By Theorem \ref{pathEk}, \[{_k}V_g \text{ is path-connected in } \mathcal{F} \text{ for all } g\in VMO_{\mathbb{R}} \text{ and for all } k\in\mathbb{Z}.\]  By Theorem \ref{F}, \eqref{firstDefEk} and the fact that $\mathcal{I}=\bigcup_{k\in\mathbb{Z}}I_k$, $\mathcal{F}=\bigcup_{k\in\mathbb{Z}} E_k$.  Thus another application of Theorem \ref{pathEk} yields \[\mathcal{F}=\bigcup_{\left\{\begin{subarray}{l} k\in\mathbb{Z} \text{ and either } \\ g\in VMO_{\mathbb{R}}\setminus L_\mathbb{R}^\infty \text{ or }\\ g=0 \end{subarray}\right\}} {_k}V_g .\]

Suppose for some $f\in\mathcal{I}$ and $K\in\mathfrak{LC}(H^2)$, $T_f+K\in {_k}V_y\cap{_m}V_h$ for some ${_k}V_y,{_m}V_h\in\{{_k}V_g\}_{\left\{\begin{subarray}{l} k\in\mathbb{Z} \text{ and either } \\ g\in VMO_{\mathbb{R}}\setminus L_\mathbb{R}^\infty \text{ or } \\ g=0\end{subarray}\right\}}$.  Then for some $q\in \chi_{-k}Q_y$ and $L\in\mathfrak{LC}(H^2)$ we have $T_f+K=T_q+L$.  Then by Lemma \ref{compactToe} $f=q$ and $f\in \chi_{-k}Q_y$.  Similarly $f\in \chi_{-m}Q_h$.  Then by Theorem \ref{mainI}, $\chi_{-k}Q_y=\chi_{-m}Q_h$.  It follows that ${_k}V_y={_m}V_h$ and $\{{_k}V_g\}_{\left\{\begin{subarray}{l} k\in\mathbb{Z} \text{ and either } \\ g\in VMO_{\mathbb{R}}\setminus L_\mathbb{R}^\infty \text{ or } \\ g=0\end{subarray}\right\}}$ is pairwise-disjoint.

Suppose $W$ is a nonempty path-connected subspace of $\mathcal{F}$ such that $W\cap {_k}V_y\neq \emptyset$ and $W\cap {_m}V_h\neq\emptyset$ for some ${_k}V_y,{_m}V_h\in\{{_k}V_g\}_{\left\{\begin{subarray}{l} k\in\mathbb{Z} \text{ and either } \\ g\in VMO_{\mathbb{R}}\setminus L_\mathbb{R}^\infty \text{ or } \\ g=0\end{subarray}\right\}}.$  Let $f_y\in Q_y$ and $f_h\in Q_h$.  Then since $T_{\chi_{-k}f_y}+K\in {_k}V_y$ and $T_{\chi_{-m}f_h}+K\in{_m}V_h$ where $K\in\mathfrak{LC}(H^2)$, $T_{\chi_{-k}f_y}+K$ is path-connected to $T_{\chi_{-m}f_h}+K$ in $\mathcal{F}$.  Then since $j$ is continuous on $\mathcal{F}$ and integer valued, $j(T_{\chi_{-k}f_y}+K)=j(T_{\chi_{-m}f_h}+K)$.  That is $m=k$.  Then by Theorem \ref{pathEk}, ${_k}V_y= {_m}V_h$.  Therefore by Lemma \ref{equiv}, $\{{_k}V_h\}_{\left\{\begin{subarray}{l} k\in\mathbb{Z} \text{ and either } \\ g\in VMO_{\mathbb{R}}\setminus L_\mathbb{R}^\infty \text{ or } \\ g=0\end{subarray}\right\}}$ are the path-connected components of $\mathcal{F}$.

\end{proof}

\section{Uncountably many path-connected components of $\mathcal{F}$}

\begin{cor}\label{uncountF}
$\mathcal{F}$ has uncountably many path-connected components.
\end{cor}

\begin{proof}
Let $k\in\mathbb{Z}$ and assume for some $h,g\in VMO_\mathbb{R}$, ${_k}V_{h}={_k}V_{g}$.  Then by Lemma \ref{compactToe}, $\chi_{-k}Q_{h}=\chi_{-k}Q_{g}$.  Let $H$ be as in Proposition \ref{Ex}.  It follows from the proof of Theorem \ref{uncountI} that \[\beta\neq\gamma \text{ implies } {_k}V_{\beta H}\neq{_k}V_{\gamma H}\] for all $\beta,\gamma\in\mathbb{R}$.  Hence $\{{_k}V_g\}_{\left\{\begin{subarray}{l}g\in VMO_{\mathbb{R}}\setminus L_\mathbb{R}^\infty \text{ or } \\ g=0\end{subarray}\right\}}$ is uncountable.  So $E_k$ has uncountably many path-connected components.  Therefore $\mathcal{F}$ has uncountably many path-connected components.
\end{proof}

\begin{rem}
In fact since $k$ in the proof of Corollary \ref{uncountF} was arbitrary, $E_k$ has uncountable many path-connected components for all $k\in\mathbb{Z}$.

\end{rem}

\subsection*{Acknowledgment}
The author would like to thank his dissertation committee for their help in writing this article: Dr. Jingbo Xia, Dr. Lewis Coburn and Dr. Ching Chou.

\end{document}